\def\benm{\begin{enumerate}}
\def\eenm{\end{enumerate}}

\def\bal{\begin{align}}
\def\eal{\end{align}}

\documentclass{amsart}
\usepackage{amsmath, amsthm, amscd, amsfonts,amssymb, graphicx}

\newtheorem{theorem}{Theorem}[section]

\theoremstyle{definition}
\newtheorem{definition}[theorem]{Definition}

\newtheorem{proposition}[theorem]{Proposition}
\newtheorem{corollary}[theorem]{Corollary}

\theoremstyle{remark}
\newtheorem{remark}[theorem]{Remark}

\numberwithin{equation}{section}

\begin{document}

\title[Measure Algebras on Coset Spaces of Compact Subgroups]
{Abstract Structure of Measure Algebras on Coset Spaces of Compact Subgroups in Locally Compact Groups}

\author[A. Ghaani Farashahi]{Arash Ghaani Farashahi}
\address{Department of Pure Mathematics, School of Mathematics, University of Leeds, Leeds LS2 9JT, United Kingdom.}
\email{a.ghaanifarashahi@leeds.ac.uk}
\email{ghaanifarashahi@outlook.com}

\curraddr{}





\date{}


\thanks{E-mail addresses: a.ghaanifarashahi@leeds.ac.uk (Arash Ghaani Farashahi)}

\begin{abstract}
This paper presents a systematic operator theory approach for abstract structure of 
Banach measure algebras 
over coset spaces of compact subgroups. Let $H$ be a compact subgroup of a locally compact group $G$ and $G/H$ be the left coset space associated to the subgroup $H$ in $G$. Also, let $M(G/H)$ be the Banach measure space consists of all complex measures over $G/H$. We then introduce an operator theoretic characterization for the abstract notion of involution over the Banach measure space $M(G/H)$. 
\end{abstract}

\maketitle

\section{\bf{Introduction}}

The following paper extends our recent results concerning the abstract structure of involutions on some Banach algebras on homogeneous spaces \cite{AGHF.BBMS-NS, AGHF.IntJM, AGHF.IJM} to more general settings. The mathematical theory of Banach convolution algebras plays significant and classical roles 
in abstract harmonic analysis, representation theory, functional analysis, operator theory, and $C^*$-algebras, see \cite{der00, Dix, HR1, HR2, kan.lau, Lau.ADV.2016, Lau.Ri, Lau.Ngai, Lau.ADV.2012, 50} and references therein. Over the last decades, some new aspects and applications of Banach convolution algebras have achieved significant popularity in different areas such as constructive approximation \cite{Fei0, Fei1, Fei2}, and theoretical aspects of coherent state (covariant) analysis, see \cite{kisil.cov1} and references therein. Homogeneous spaces are group-like structures with many applications in mathematical physics, differential geometry, geometric analysis, and coherent state (covariant) transforms, see \cite{AGHF.CAOT, AGHF.AAMP, AGHF.FUM, kisil, kisil.cov2, kisil1}.

Let $G$ be a locally compact group and $H$ be a compact subgroup of $G$. 
Suppose $G/H$ is the left coset space of the subgroup $H$ in $G$.
In \cite{Kam.Tav} that author's attitude toward rigor and precision is exceedingly carefree and the paper is full of inaccuracies, they tried to present mathematical definitions for the notions of convolution and involution for functions defined on $G/H$, which did not work. 
In Section 3 of \cite{AGHF.BMMSS} as a part of the paper, Ghaani Farashahi showed that the presented definition in \cite{Kam.Tav} is well-defined only when $H$ is normal in $G$. In this case, the introduced convolution and involution are the same as the canonical convolution and involution on the quotient groups, that is not the case for (pure) homogeneous spaces.
The solid and well-defined definitions for the abstract notion of convolution and different approaches to the abstract notion of involution on coset spaces of compact subgroups in locally compact groups introduced in (3.5) and (4.3) of \cite{AGHF.BMMSS}, respectively.

A unified operator theoretic approach introduced by Ghaani Farashahi, which gives a better understanding and characterization of the algebraic (convolution and involution) structure in both function and measure settings on the coset space $G/H$. In the case $G$ is compact, analytic and algebraic aspects of this operator theoretic approach for functions studied at depth in \cite{AGHF.CJM, AGHF.IJM} and in the settings of measures investigated extensively in \cite{AGHF.IntJM}. Our results concerning operator theoretic characterization for structure of convolution and involution of functions defined on homogeneous spaces of compact groups \cite{AGHF.IJM}, extended to the case $G$ is locally compact and $H$ is compact in \cite{AGHF.BBMS-NS}.

This paper extends our results in two different directions. As the first direction, 
this paper extends
results of \cite{AGHF.IntJM} related to involutions on homogenuos spaces of compact groups, to the case $G$ is locally compact group and $H$ is a compact subgroup. Second direction is extension of operator theoretic characterization of function spaces presented in \cite{AGHF.BBMS-NS} for measure spaces.
This article contains 5 sections. Section 2 is devoted to fixing notations and preliminaries on functional analysis, classical mathematical analysis on locally compact Hausdorff spaces, and classical harmonic analysis on locally compact groups and (left) coset spaces of compact subgroups. Let $\mu$ be the normalized $G$-invariant measure on the homogeneous space $G/H$ associated to  Weil's formula. Also, let 
$M(G/H)$ be the Banach measure space consists of all complex measures over $G/H$.
In section 3, a short review for operator theoretic characterization of the abstract notions of generalized convolution and involution on classical function spaces over the homogeneous space $G/H$ is given. Section 4 is devoted to present a unified summary for classical aspects of abstract harmonic analysis over spaces of complex measures on coset spaces of compact subgroups.  Section 5 is proposed to introduce the operator theoretic characterization for the abstract notions of involution over the Banach measure space $M(G/H)$. 

\section{\bf{Preliminaries and Notations}}

Throughout this section, we present basic preliminaries and notations.

\subsection{Functional and Classical Analysis over Locally Compact Hausdorff Spaces}
Let $\mathcal{X},\mathcal{Y}$ be Banach spaces with the topological (linear) dual spaces $\mathcal{X}^*$ 
and $\mathcal{Y}^*$ respectively. 
Also, let $T:\mathcal{X}\to\mathcal{Y}$ be a bounded linear operator. 
The Banach space adjoint of the linear map 
$T$, is the bounded linear map $T^*:\mathcal{Y}^*\to\mathcal{X}^*$ given by 
\[
T^*(\beta)(a):=\beta(T(a)),
\]
for all $a\in X$ and $\beta\in\mathcal{Y}^*$.

For a locally compact Hausdorff space $X$, 
$\mathcal{C}(X)$ denotes the space of all continuous complex-valued functions 
on $X$, $\mathcal{C}_0(X)$ stands for the subspace of $\mathcal{C}(X)$ consists of all 
continuous complex-valued functions on $X$ which are vanishing at infinity, and $\mathcal{C}_c(X)$ is the subspace 
of $\mathcal{C}(X)$ which contains all continuous complex-valued functions 
on $X$ with compact support. It is easy to see that 
\[
 \mathcal{C}_c(X)\subseteq\mathcal{C}_0(X)\subseteq\mathcal{C}(X).
\]

If $\mu$ is a positive Radon measure on $X$, for each $1\le p<\infty$ 
the Banach space of equivalence classes of $\mu$-measurable complex valued functions $f:X\to\mathbb{C}$ such that
$$\|f\|_{L^p(X,\mu)}=\left(\int_X|f(x)|^pd\mu(x)\right)^{1/p}<\infty,$$
is denoted by $L^p(X,\mu)$ which contains $\mathcal{C}_c(X)$ as a $\|.\|_{L^p(X,\mu)}$-dense subspace.
The linear space consists of all regular countably additive complex Borel measures on $X$ is denoted 
by $M(X)$. It is a Banach space with respect to the total variation norm $\|.\|_{M(X)}$,
defined as 
\[
\|\lambda\|_{M(X)}:=|\lambda|(X), 
\]
for all $\lambda\in M(X)$, where $|\lambda|$ is the absolute value of $\lambda$.

It is well-known as the Riesz-Markov theorem \cite{HR1} that, (i) for any unique regular 
countably additive complex Borel measure $\lambda$ on $X$,
the mapping $f\mapsto\lambda(f)$ is a continuous functional on $\mathcal{C}_0(X)$, where 
\[
\lambda(f):=\int_Xf(x)d\lambda(x).
\]
(ii) for any continuous linear functional $\Gamma$ on $\mathcal{C}_0(X)$, 
there is a unique regular countably additive complex Borel measure $\lambda_\Gamma$ on $X$ such that
\[
\Gamma(f)=\lambda_\Gamma(f)=\int_X f(x)d\lambda_\Gamma(x), 
\]
for all $f\in\mathcal{C}_0(X)$. Also, we have 
$\|\Gamma\|=\|\lambda_\Gamma\|_{M(X)}$, where $\|\Gamma\|$ is the operator norm of the functional $\Gamma$, that is 
\[
\|\Gamma\|:=\sup_{\{f\in\mathcal{C}_0(X):\|f\|_{\sup}\le 1\}}|\Gamma(f)|. 
\]

\subsection{Abstract Harmonic Analysis over Locally Compact Groups}
Let $G$ be a locally compact group with the modular function $\Delta_G$ and a 
left Haar measure $\sigma$. 
For $p\ge 1$ the notation $L^p(G)$ stands for the Banach function space $L^p(G,\sigma)$. 
The standard convolution for $f,g\in L^1(G)$, is defined via
$$
f\ast_G g(x)=\int_Gf(y)g(y^{-1}x)d\sigma(y)\quad (x\in G).
$$
The involution for $f\in L^1(G)$, is defined by $f^{*^G}(x)=\Delta_G(x^{-1})\overline{f(x^{-1})}$ for $x\in G$. 
Then the Banach function space $L^1(G)$ equipped with the convolution (\ref{Lp.G}) 
and also the involution (\ref{inv.G}) is 
a Banach $*$-algebra, that is 
\begin{equation}\label{Lp.G}
\|f\ast_G g\|_{1}\le \|f\|_1\|g\|_1,
\end{equation}
and 
\begin{equation}\label{inv.G}
(f\ast_G g)^{*^G}=g^{*^G}\ast_G f^{*^G},
\end{equation}
for all $f,g\in L^1(G)$, see \cite{HR1, HR2, 50} and classical list of references therein.

The standard convolution for complex measures 
$\nu,\nu'\in M(G)$ is the complex measure $\nu\ast_G\nu'\in M(G)$
which is given by 
\begin{equation}\label{m.conv.G}
\nu\ast_G\nu'(f)=\int_G\int_G f(xy)d\nu(x)d\nu'(y),
\end{equation}
for all $f\in\mathcal{C}_0(G)$. Also, the involution of the complex measure 
$\nu\in M(G)$ is the complex measure $\nu^{\ast^G}$ given by 
\begin{equation}\label{m.involution.G}
\nu^{\ast^G}(f):=\int_Gf(x^{-1})d\overline{\nu}(x),
\end{equation}
for all $f\in\mathcal{C}_0(G)$.

The Banach measure space $M(G)$ equipped with the convolution (\ref{m.conv.G}) 
and also the involution (\ref{m.involution.G}) is 
a Banach $*$-algebra, that is 
\begin{equation}\label{m.G}
\|\nu\ast_G\nu'\|_{M(G)}\le \|\nu\|_{M(G)}\|\nu'\|_{M(G)},
\end{equation}
and 
\begin{equation}\label{m.conv.inv.G}
(\nu\ast_G \nu')^{*^G}={\nu'}^{*^G}\ast_G \nu^{*^G},
\end{equation}
for all $\nu,\nu'\in M(G)$, see \cite{HR1, HR2, 50} and classical list of references therein.

For $f\in L^1(G)$, let $\sigma_f\in M(G)$ be the complex measure on $G$ given by 
\[
\sigma_f(g)=\int_Gg(x)f(x)d\sigma(x),
\]
for all $g\in\mathcal{C}_0(G)$. Then, it is well-known that $f\mapsto\sigma_f$ 
is an isometric $*$-homomorphism embedding of the Banach function 
$*$-algebra $L^1(G)$ into the Banach measure $*$-algebra $M(G)$, see \cite{HR1, 50}. That is, 
\begin{equation}\label{sigma.hom}
\|\sigma_f\|_{M(G)}=\|f\|_{L^1(G)},\hspace{1cm}\sigma_{f\ast_Gf'}=\sigma_f\ast_G\sigma_{f'},
\end{equation}
for all $f,f'\in L^1(G)$.

Let $H$ be a compact subgroup of the locally compact group $G$ with the probability Haar measure $dh$. 
The left coset space $G/H$ is considered as a locally compact homogeneous space that $G$ 
acts on it from the left, and $q:G\to G/H$ given by $x\mapsto q(x):=xH$ is the surjective canonical map. 
The classical aspects of abstract harmonic analysis on locally compact homogeneous spaces are quite well studied in many references, see \cite{HR1, HR2, 50} and references therein.
The function space $\mathcal{C}_c(G/H)$ consists of all functions $T_H(f)$, where 
$f\in\mathcal{C}_c(G)$ and
\begin{equation}\label{5.1}
T_H(f)(xH)=\int_Hf(xh)dh.
\end{equation}
Let $\mu$ be a Radon measure on $G/H$ and $x\in G$. 
The translation $\mu_x$ of $\mu$ is defined by $\mu_x(E)=\mu(xE)$, for all Borel subsets $E$ of $G/H$. 
The measure $\mu$ is called $G$-invariant if $\mu_x=\mu$, for all $x\in G$.
The homogeneous space $G/H$ has a normalized $G$-invariant measure $\mu$, which satisfies 
Weil's formula  
\begin{equation}\label{TH.m}
\int_{G/H}T_H(f)(xH)d\mu(xH)=\int_Gf(x)dx,
\end{equation}
and hence the linear map $T_H$ is norm-decreasing, that is 
$$\|T_H(f)\|_{L^1(G/H,\mu)}\le \|f\|_{L^1(G)},$$
for all $f\in L^1(G)$. 

For a function $\varphi\in L^p(G/H,\mu)$ and $z\in G$, the left action of $z$ on $\varphi$ is defined by $L_z\varphi(xH)=\varphi(z^{-1}xH)$ for $xH\in G/H$.

\section{\bf{Abstract Structure of Function Algebras on Coset Spaces of Compact Subgroups in Locally Compact Groups}}

Throughout this paper, we assume that $G$ is a locally compact group with the fixed left Haar measure $d\sigma(x)=dx$, $H$ is a compact subgroup of $G$ with the probability Haar measure $dh$, and $\mu$ is the normalized $G$-invariant measure on the locally compact homogeneous space $G/H$ satisfying (\ref{TH.m}). 

First, we review some results concerning classical aspect of the function space 
$\mathcal{C}_0(G/H)$. These classical aspects concerning the fucntion space $\mathcal{C}_0(G/H)$ studied for the case $G$ is comapct in \cite{AGHF.BKMS} and extended for the case $G$ is locally comapct and $H$ is compact in \cite{AGHF.AM}. 

For a function $\psi:G/H\to\mathbb{C}$, define $\psi_q:G\to\mathbb{C}$ by 
$$\psi_q(x):=\psi\circ q(x)=\psi(xH),$$
for all $xH\in G/H$.

It is straightforward to check that, if $\psi\in\mathcal{C}_c(G/H)$ we then have 
$\psi_q\in\mathcal{C}_c(G)$, with $T_H(\psi_q)=\psi$, and $\|\psi\|_{\sup}=\|\psi_q\|_{\sup}$.

\begin{theorem}\label{u.c.t}
{\it Let $H$ be a compact subgroup of a locally compact group $G$.
The linear map $T_H:\mathcal{C}_c(G)\to\mathcal{C}_c(G/H)$ is uniformly norm-decreasing and 
hence, it has e a unique extension to a uniformly norm-decreasing map 
from $\mathcal{C}_0(G)$ onto $\mathcal{C}_0(G/H)$. 
}\end{theorem}
\begin{proof}
See Theorem 3.1 of \cite{AGHF.AM}.
\end{proof}
\begin{remark}
If $\psi\in\mathcal{C}_0(G/H)$ we then have 
$\psi_q\in\mathcal{C}_0(G)$, with $T_H(\psi_q)=\psi$, and $\|\psi\|_{\sup}=\|\psi_q\|_{\sup}$, see Corollary 3.2 of \cite{AGHF.AM}.
\end{remark}
\begin{remark}
It is also shown that the linear map 
$T_H:\mathcal{C}_c(G)\to\mathcal{C}_c(G/H)$ is norm-decreasing in $L^p$-sense, that is 
\begin{equation}
\|T_H(f)\|_{L^p(G/H,\mu)}\le \|f\|_{L^p(G)},
\end{equation}
for all $f\in\mathcal{C}_c(G)$ and $p>1$. Therefore, the linear map 
$T_H$ has a unique extension to a norm-decreasing
linear map in $L^p$-sense, denoted by $T_H:L^p(G)\to L^p(G/H,\mu)$, see Proposition 2.1 of \cite{AGHF.JSI}. By applying a canonical normalization of the linear operator $T_H$ given by (\ref{5.1}), most results of \cite{AGHF.JSI} concerning analytic properties of the linear map $T_H$ extended to the case that the measure $\mu$ is not $G$ invariant, see \cite{TRO}.
Authors of \cite{TRO} did not clearly mention that they exteneded the results of Ghaani Farashahi in \cite{AGHF.BMMSS, AGHF.JSI}. Instead, the authors cited \cite{AGHF.JSI} in the introduction with an incomplete statement about the content therein. 
\end{remark}
\begin{remark}\label{TH*}
{\it Let $G$ be a locally compact group and $H$ be a compact subgroup of $G$.
Let $T_H^*:\mathcal{C}_0(G/H)^*\to\mathcal{C}_0(G)^*$ be the Banach space adjoint of the linear map 
$T_H:\mathcal{C}_0(G)\to\mathcal{C}_0(G/H)$ given by Theorem \ref{u.c.t}. Then 
\[
T_H^*(\beta)=\beta_q,\ \forall \beta\in\mathcal{C}_0(G/H)^*,
\]
where $\beta_q:\mathcal{C}_0(G)\to\mathbb{C}$ is given by 
\[
\beta_q(f):=\beta(T_H(f)),\ \forall f\in\mathcal{C}_0(G).
\]
}\end{remark}

We then continue this section by a summary of our recent results concerning operator theoretic characterization of the abstract structure of function $*$-algebras over the left coset space $G/H$, we refer the readers to \cite{AGHF.BBMS-NS, AGHF.BMMSS} for details and proofs.

Suppose
\[
\mathcal{C}_c(G:H):=\{f\in\mathcal{C}_c(G):R_hf=f\ \forall h\in H\}.
\] 
Then, one can define 
\[
A(G:H):=\{f\in\mathcal{C}_c(G):L_hf=f\ \forall h\in H\},
\]
and 
\[
A(G/H):=\{\varphi\in\mathcal{C}_c(G/H):L_h\varphi=\varphi\ \forall h\in H\}.
\]
Also, let 
\[
\mathcal{C}_0(G:H):=\{f\in\mathcal{C}_0(G):R_hf=f\ \forall h\in H\},
\]
\[
A_0(G:H):=\{f\in\mathcal{C}_0(G):L_hf=f\ \forall h\in H\},
\]
and 
\[
A_0(G/H):=\{\varphi\in\mathcal{C}_0(G/H):L_h\varphi=\varphi\ \forall h\in H\}.
\]

For $1\le p\le\infty$, one can define  
\[
A^p(G:H):=\{f\in L^p(G):L_hf=f\ \forall h\in H\},
\] 
and also
\[
A^p(G/H,\mu):=\{\varphi\in L^p(G/H,\mu):L_h\varphi=\varphi\ \forall h\in H\}
\]
It is easy to see that $A^p(G/H,\mu)$ is the topological closure of $A(G/H)$ in $L^p(G/H,\mu)$ and hence it is a closed linear subspace of $L^p(G/H,\mu)$. Also, one can readily check that $A^p(G:H)$ is the topological closure of $A(G:H)$ in $L^p(G)$ and hence 
it is a closed linear subspace of $L^p(G)$. It can be checked that 
$A(G/H)=\mathcal{C}_c(G/H)$ and $A^p(G/H,\mu)=L^p(G/H,\mu)$, if $H$ is normal in $G$.

For $\psi\in\mathcal{C}_c(G/H)$, let $J\psi:G/H\to\mathbb{C}$ be given by 
\[
J\psi(xH):=\int_H\psi(hxH)dh,
\]
for all $xH\in G/H$.

Then $J:\mathcal{C}_c(G/H)\to\mathcal{C}_c(G/H)$ given by $\psi\mapsto J\psi$ is a 
linear operator. The linear operator $J:\mathcal{C}_c(G/H)\to\mathcal{C}_c(G/H)$ 
is the identity operator if $H$ is normal in $G$, 
see Remark 4.4 of \cite{AGHF.BBMS-NS}.  

The following results present basic properties of the linear operator $J$ in the framework of the function space $\mathcal{C}_0(G/H)$.

\begin{theorem}
{\it Let $G$ be a locally compact group and $H$ be a compact subgroup of $G$.
Let $\mu$ be the normalized $G$-invariant measure over the left coset space $G/H$ 
associated to Weil's formula. Then
\begin{enumerate}
\item For $\psi\in\mathcal{C}_c(G/H)$ we have 
$\|J\psi\|_{\sup}\le \|\psi\|_{\sup}$.
\item $J$ maps $\mathcal{C}_c(G/H)$ onto $A(G/H)$.
\end{enumerate}
}\end{theorem}
\begin{proof}
(1) Let $\psi\in\mathcal{C}_c(G/H)$ be given. Since $H$ is compact, we have 
\begin{align*}
|J\psi(xH)|
&=\left|\int_H\psi(hxH)dh\right|
\\&\le\int_H|\psi(hxH)|dh
\\&\le\int_H\|\psi\|_{\sup}dh
\\&=\|\psi\|_{\sup}\left(\int_Hdh\right)=\|\psi\|_{\sup},
\end{align*}
for all $x\in G$. Thus, we get 
\[
\|J\psi\|_{\sup}=\sup_{xH\in G/H}|J\psi(xH)|\le\|\psi\|_{\sup}.
\]
(2) See Theorem 4.5(2) of \cite{AGHF.BBMS-NS}.
\end{proof}

\begin{corollary}\label{J0}
{\it Let $G$ be a locally compact group and $H$ be a compact subgroup of $G$.
The linear operator $J:\mathcal{C}_c(G/H)\to A(G/H)$ has a unique extension to the surjective norm-decreasing linear operator $J_0:\mathcal{C}_0(G/H)\to A_0(G/H)$.
}\end{corollary} 

\begin{remark}\label{eJ.normal.0}
Let $G$ be a locally compact group and $H$ be a compact subgroup of $G$. 
The extended linear operator $J_0:\mathcal{C}_0(G/H)\to A_0(G/H)$ 
is the identity operator.  
\end{remark}

\begin{remark} 
Let $\mu$ be the normalized $G$-invariant measure over the left coset space 
$G/H$ associated to Weil's formula and $1\le p\le\infty$. It is shown that each 
$\psi\in\mathcal{C}_c(G/H)$, satisfies (Theorem 4.5(1) of \cite{AGHF.BBMS-NS})
\[
\|J\psi\|_{L^p(G/H,\mu)}\le \|\psi\|_{L^p(G/H,\mu)}.
\]
Therefore, the linear operator $J:\mathcal{C}_c(G/H)\to A(G/H)$ has a unique extension to 
a surjective and norm-decreasing linear operator, denoted by 
$$J_p:L^p(G/H,\mu)\to A^p(G/H,\mu),$$
see Corollary 4.6 of \cite{AGHF.BBMS-NS}.
\end{remark}

\begin{definition}
Let $G$ be a locally compact group and $H$ be a compact subgroup of $G$. 
For $\varphi,\psi\in\mathcal{C}_c(G/H)$, let $\varphi\ast_{G/H}\psi:G/H\to\mathbb{C}$ be given by 
\begin{equation}\label{conv.c}
\varphi\ast_{G/H}\psi(xH)=\int_{G/H}\varphi(yH)J\psi(y^{-1}xH)d\mu(yH),
\end{equation}
for all $xH\in G/H$.
\end{definition}
It can be checked that the convolution defined by (\ref{conv.c}) 
coincides with the canonical convolution over the quotient group $G/H$, if $H$ is normal in $G$, see \cite{50}.

Next we state some algebraic results of \cite{AGHF.BBMS-NS} concerning 
basic properties for convolution of functions. 

\begin{proposition}\label{conv.basic}
{\it Let $G$ be a locally compact group and $H$ be a compact subgroup of $G$. 
Also, let $\varphi,\psi\in\mathcal{C}_c(G/H)$. We then have 
\begin{enumerate}
\item $\varphi\ast_{G/H}\psi=T_H(\varphi_q\ast_G\psi_q)$.
\item $(\varphi\ast_{G/H}\psi)_q=\varphi_q\ast_G\psi_q$.
\item $L_z(\varphi\ast_{G/H}\psi)=(L_z\varphi)\ast_{G/H}\psi$.
\end{enumerate}
}\end{proposition} 

\begin{remark}
Proposition \ref{conv.basic}(1) shows that the convolution defined by (\ref{conv.c}) is an operator theoretic characterization for convolution of functions introduced in (3.5) of \cite{AGHF.BMMSS}, when the measure $\mu$ is $G$-invariant.
\end{remark}

\begin{remark}
By employing a canonical normalization of the linear operator $T_H$ given by (\ref{5.1}), the convolution of functions introduced by (3.5) of \cite{AGHF.BMMSS} extended to the case that the measure $\mu$ is not relatively invariant, see \cite{Jav.Tav, TRO}. Authors of \cite{TRO} did not cite \cite{AGHF.BMMSS}. Also, authors of \cite{Jav.Tav} did not cite \cite{AGHF.BMMSS}. Theorem 5.2 of \cite{Jav.Tav} is a canonical extension for the convolution introduced by (3.5) of \cite{AGHF.BMMSS}. 
\end{remark}

The function $\varphi\ast_{G/H}\psi$ is called as convolution of $\varphi$ and $\psi$.
It is easy to check that the map $$\ast_{G/H}:\mathcal{C}_c(G/H)\times\mathcal{C}_c(G/H)\to\mathcal{C}_c(G/H),$$ 
given by 
$$(\varphi,\psi)\mapsto \varphi\ast_{G/H}\psi,$$ 
is bilinear. 

Also, it can be readily see that the linear space $\mathcal{C}_c(G/H)$ with respect to $\ast_{G/H}$ as multiplication, is an associative algebra.

We then review some of the following results of \cite{AGHF.BBMS-NS} concerning 
basic analytic properties for convolution of functions given by (\ref{conv.c}). 
 
\begin{theorem}\label{main.conv}
Let $G$ be a locally compact group and $H$ be a closed subgroup of $G$.
Let $\mu$ be the normalized $G$-invariant measure on $G/H$. 
\begin{enumerate}\label{norm.algebra}
\item For $\varphi,\psi\in\mathcal{C}_c(G/H)$, we have 
\begin{equation}
\|\varphi\ast_{G/H}\psi\|_{L^1(G/H,\mu)}\le\|\varphi\|_{L^1(G/H,\mu)}\|\psi\|_{L^1(G/H,\mu)}.
\end{equation}
\item The convolution map $\ast_{G/H}:\mathcal{C}_c(G/H)\times\mathcal{C}_c(G/H)\to\mathcal{C}_c(G/H)$ given by (\ref{conv.c}) has a unique extension to 
$$\ast_{G/H}:L^1(G/H,\mu)\times L^1(G/H,\mu)\to L^1(G/H,\mu),$$ 
in which the Banach function space $L^1(G/H,\mu)$ equipped with the extended convolution is a Banach algebra. 
\item For $\varphi,\psi\in L^1(G/H,\mu)$ and $x\in G$, we 
have the following explicit characterization 
\[
\varphi\ast_{G/H}\psi(xH)=\int_{G/H}\varphi(yH)J_1\psi(y^{-1}xH)d\mu(yH).
\]
\item $A^1(G/H,\mu)$ is a Banach function sub-algebra of $L^1(G/H,\mu)$.
\end{enumerate}
\end{theorem}

\begin{definition}
Let $G$ be a locally compact group and $H$ be a compact subgroup of $G$. 
For $\varphi\in\mathcal{C}_c(G/H)$, let $\varphi^{\ast_{G/H}}:G/H\to\mathbb{C}$ be given by 
\begin{equation}\label{inv.c}
\varphi^{\ast_{G/H}}(xH)=\Delta_G(x^{-1})\int_H\overline{\varphi(h^{-1}x^{-1}H)}dh,
\end{equation}
for all $xH\in G/H$.
\end{definition}

The function $\varphi^{\ast_{G/H}}$ is called as involution of $\varphi$.
It is easy to check that the map 
$$^{\ast_{G/H}}:\mathcal{C}_c(G/H)\to\mathcal{C}_c(G/H),$$ 
given by $\varphi\mapsto\varphi^{\ast_{G/H}}$ is conjugate linear. 
Also, involution defined by (\ref{inv.c}) 
coincides with the canonical involution over the locally compact quotient group $G/H$ if $H$ is normal in $G$. 

We hereby finish this section by reviewing the following results of \cite{AGHF.BBMS-NS}. 

\begin{proposition}\label{inv.basic}
{\it Let $G$ be a locally compact group and $H$ be a compact subgroup of $G$.
Let $\varphi\in\mathcal{C}_c(G/H)$ and $\mu$ be the normalized $G$-invariant measure on 
the left coset space $G/H$. Then we have 
\begin{enumerate}
\item $\varphi^{\ast^{G/H}\ast^{G/H}}=J\varphi$.
\item $\varphi^{\ast_{G/H}}=T_H(\varphi_q^{\ast_G})$.
\item $\|\varphi^{\ast_{G/H}}\|_{L^1(G/H,\mu)}\le\|\varphi\|_{L^1(G/H,\mu)}$.
\end{enumerate}
In particular, if $\varphi\in A(G/H)$, we then have 
\begin{enumerate}
\item $\varphi^{\ast^{G/H}\ast^{G/H}}=\varphi$.
\item $\|\varphi^{\ast_{G/H}}\|_{L^1(G/H,\mu)}=\|\varphi\|_{L^1(G/H,\mu)}$.
\item $(\varphi^{\ast_{G/H}})_q=\varphi_q^{\ast_G}$.
\end{enumerate}
}\end{proposition}

\begin{proposition}
{\it Let $G$ be a locally compact group and $H$ be a compact subgroup of $G$. 
Also, let $\varphi,\psi\in\mathcal{C}_c(G/H)$. We then have 
\[
(\varphi\ast_{G/H}\psi)^{\ast^{G/H}}=\psi^{\ast^{G/H}}\ast_{G/H}\varphi^{\ast^{G/H}}.
\] 
Therefore, the linear space $A(G/H)$ equipped with the convolution $\ast_{G/H}$ 
and the involution $^{\ast_{G/H}}$ is a normed $*$-algebra. In particular, 
the Banach function algebra $A^1(G/H,\mu)$ equipped with the extended involution 
is a Banach function $*$-algebra.
}\end{proposition}

\section{\bf{Abstract Harmonic Analysis on Spaces of Complex Measures on Coset 
Spaces of Compact Subgroups in Locally Compact Groups}}

Throughout this section, we review some of basic results concerning abstract harmonic 
analysis over spaces of complex measures on coset spaces of compact subgroups in locally compact groups, for details and proofs see \cite{AGHF.AM, AGHF.BKMS}. Also, it is still assumed that $G$ is a locally compact group, $H$ is a compact subgroup of $G$, and 
$\mu$ is the normalized $G$-invariant measure over the left coset space $G/H$ associated to (\ref{TH.m}) with respect to the fixed left Haar measure $dx=d\sigma(x)$ of $G$ and the 
probability measure of $H$. 
It should be mentioned that, from now on by a complex measure we mean a regular countably additive complex 
Borel measure. 

For a complex measure $\nu\in M(G)$, let $T_H(\nu)\in M(G/H)$ be the complex measure 
which satisfies  
\begin{equation}\label{TH.M.def}
 \int_{G/H}\psi(xH)dT_H(\nu)(xH)=\int_G\psi_q(x)d\nu(x), 
\end{equation}
for all $\psi\in\mathcal{C}_0(G/H)$. 
Then, $T_H:M(G)\to M(G/H)$ given by $\nu\mapsto T_H(\nu)$ is a surjective linear map. 
It is norm-decreasing as well, that is 
\begin{equation}\label{TH.M.norm}
\|T_H(\nu)\|_{M(G/H)}\le \|\nu\|_{M(G)},
\end{equation}
for all $\nu\in M(G)$, see \cite{50}.

Let $\lambda\in M(G/H)$ be a complex measure. Then, $\Gamma_{\lambda}:\mathcal{C}_0(G)\to\mathbb{C}$
given by 
\[
f\mapsto\Gamma_\lambda(f):=\int_{G/H}T_H(f)(xH)d\lambda(xH),
\]
is a linear functional. Also, it is continuous. Because, we have 
$$|\Gamma_\lambda(f)|
\le\|f\|_{\sup}\|\lambda\|_{M(G/H)}.$$ 
Thus, invoking Riesz-Markov theorem, there exists a unique complex measure, denoted by $\lambda_q\in M(G)$, satisfying 
\begin{equation}
\int_{G}f(x)d\lambda_q(x)=\int_{G/H}T_H(f)(xH)d\lambda(xH),
\end{equation}
for all $f\in\mathcal{C}_0(G)$. 

The following results studied for the case $G$ is compact in \cite{AGHF.BKMS} and extended for the case $G$ is locally comapct and $H$ is compact in \cite{AGHF.AM}.  

\begin{proposition}\label{lambda.q.0}
{\it Let $G$ be a locally compact group and $H$ a compact subgroup of $G$.
Let $\lambda\in M(G/H)$. We then have 
\begin{enumerate}
\item $T_H(\lambda_q)=\lambda$.
\item For each $f\in\mathcal{C}_0(G)$ and $h\in H$ we have 
\[
\int_Gf(xh)d\lambda_q(x)=\int_Gf(x)d\lambda_q(x).
\]
\item Let $T_H^*:M(G/H)\to M(G)$ be the Banach space adjoint of the norm-decreasing 
linear map $T_H:\mathcal{C}_0(G)\to\mathcal{C}_0(G/H)$ given by Theorem \ref{u.c.t}. 
Then 
\begin{equation}
T_H^*(\lambda)=\lambda_q,\ \forall \lambda\in M(G/H).
\end{equation}
\end{enumerate}
}\end{proposition} 
\begin{proof}
See Proposition 4.1 and Theorem 4.2 of \cite{AGHF.AM}.
\end{proof}

Let $\mu$ be the normalized $G$-invariant measure over the left coset space $G/H$ and 
$\varphi\in L^1(G/H,\mu)$. Then, one can define the continuous linear functional
\begin{equation}\label{em.fun}
\psi\mapsto \int_{G/H}\psi(xH)\varphi(xH)d\mu(xH),
\end{equation}
for all $\psi\in\mathcal{C}_0(G/H)$. 

Let $\mu_{\varphi}$ be the complex Radon measure on the left coset space 
$G/H$ associated to the continuous linear functional given by (\ref{em.fun}). 
Thus, we get 
\begin{equation}
 \int_{G/H}\psi(xH)d\mu_\varphi(xH)=\int_{G/H}\psi(xH)\varphi(xH)d\mu(xH),
\end{equation}
for all $\psi\in\mathcal{C}_0(G/H)$.

\begin{proposition}
{\it Let $G$ be a locally compact group and $H$ be a compact subgroup of $G$.
Let $\mu$ be the normalized $G$-invariant measure over $G/H$ associated to 
(\ref{TH.m}) with respect to the left Haar measure $\sigma$ of $G$ and the 
probability measure of $H$.
\begin{enumerate}
\item For $\varphi\in L^1(G/H,\mu)$, we have 
\begin{equation}\label{q.em}
(\mu_{\varphi})_q=\sigma_{\varphi_q}.
\end{equation}
\item For $\lambda\in M(G/H)$, we have 
\begin{equation}\label{q.lam}
\|\lambda_q\|_{M(G)}=\|\lambda\|_{M(G/H)}.
\end{equation}
\end{enumerate}
}\end{proposition} 
\begin{proof}
See Propositions 4.3 and 4.4 of \cite{AGHF.AM}.
\end{proof}

\begin{theorem}
{\it Let $G$ be a locally compact group and $H$ be a compact subgroup of $G$.
Let $\mu$ be the normalized $G$-invariant measure over $G/H$ associated to (\ref{TH.m}) with respect to 
the left Haar measure $\sigma$ of $G$ and the 
probability measure of $H$. Then, 
 $\varphi\mapsto\mu_\varphi$ defines an isometric linear embedding of the Banach function space 
$L^1(G/H,\mu)$ into the Banach measure space $M(G/H)$.
}\end{theorem}
\begin{proof}
See Theorem 4.5 of \cite{AGHF.AM}.
\end{proof}

For $\nu\in M(G)$ and $g\in G$, let $\nu_g$ (resp. $\nu^g$) be the left (resp. right)
translation of $\nu$ by $g$, that is $\nu_g(E)=\nu(gE)$ (resp. $\nu^g(E)=\nu(Eg)$) for all Borel subsets $E$ of $G$.
Let $M(G:H)$ and $\mathcal{M}(G:H)$ be the linear subspaces of $M(G)$ given by 
\[
M(G:H):=\{\nu\in M(G):\nu^h=\nu\ \forall h\in H\},
\]
and 
\[
\mathcal{M}(G:H):=\{\nu\in M(G):\nu_h=\nu\ \forall h\in H\}.
\]
Also, let $\mathcal{M}(G/H)$ be the linear subspace of $M(G/H)$ given by 
\[
\mathcal{M}(G/H):=\{\lambda\in M(G/H):\lambda_h=\lambda\ \forall h\in H\},
\]
where $\lambda_x$ is the left translation of $\lambda\in M(G/H)$ 
by $x\in G$.

\begin{remark}
Let $G$ be a locally compact group and $H$ be a compact normal subgroup of $G$.
Then, we get $\mathcal{M}(G:H)=M(G)$ and also $\mathcal{M}(G/H)=M(G/H)$.
\end{remark}

We finish this section by the following straightforward observations.

\begin{proposition}\label{main.m.G}
{\it Let $G$ be a locally compact group and $H$ be a compact subgroup of $G$.
Let $\nu\in M(G)$, $\lambda,\lambda'\in M(G/H)$, and $x\in G$. 
We then have 
\begin{enumerate}
\item $T_H(\nu_x)=T_H(\nu)_x$.
\item $(\lambda_x)_q=(\lambda_q)_x$.
\end{enumerate}
}\end{proposition}

\begin{proposition}\label{main.m.G1}
{\it Let $G$ be a locally compact group and $H$ be a compact subgroup of $G$.
We then have 
\begin{enumerate}
\item $M(G:H)=\{\lambda_q:\lambda\in M(G/H)\}$.
\item $\mathcal{M}(G/H)=\{\lambda\in M(G/H):\lambda\circ J_0=\lambda\}$.
\item $T_H$ maps $M(G:H)$ onto $M(G/H)$.
\item $T_H$ maps $\mathcal{M}(G:H)$ onto $\mathcal{M}(G/H)$.
\end{enumerate}
}\end{proposition}

\section{\bf{Abstract Structure of Measure Algebras on Coset Spaces of Compact Subgroups in Locally Compact Groups}}

In this section, we study the abstract structure of measure algebras over left coset spaces of compact subgroups in locally compact groups. In this direction, we introduce the abstract notation of involution for complex measures on coset spaces of compact subgroups in locally compact groups. We then study some analytic aspects of these involutions on measure algebras. In details, this section presents canonical extensions for our results concerning 
the abstract notions of involution associated to the Banach convolution measure spaces on coset spaces of compact groups \cite{AGHF.IntJM} into a more general settings, that is the case of coset spaces of compact subgroups in locally compact groups. 

For complex Radon measures $\lambda,\lambda'\in M(G/H)$, define 
the linear functional $\Lambda_{\lambda,\lambda'}:\mathcal{C}_0(G/H)\to\mathbb{C}$ by 
\[
\psi\mapsto\Lambda_{\lambda,\lambda'}(\psi):=\int_{G/H}\int_{G/H}\left(\int_H\psi(xhyH)dh\right)d\lambda(xH)d\lambda'(yH),
\]
for all $\psi\in\mathcal{C}_0(G/H)$.

Using compactness of $H$ we can write 
\begin{align*}
|\Lambda_{\lambda,\lambda'}(\psi)|
&=\left|\int_{G/H}\int_{G/H}\left(\int_H\psi(xhyH)dh\right)d\lambda(xH)d\lambda'(yH)\right|
\\&\le\int_{G/H}\int_{G/H}\left|\int_H\psi(xhyH)dh\right|d|\lambda|(xH)d|\lambda'|(yH)
\\&\le\int_{G/H}\int_{G/H}\int_H\left|\psi(xhyH)\right|dhd|\lambda|(xH)d|\lambda'|(yH)
\le \|\psi\|_{\sup}\|\lambda\|\|\lambda'\|.
\end{align*}

Thus, by Riesz-Markov theorem, 
there exists a complex Radon measure, denoted by 
$\lambda\ast_{G/H}\lambda'\in M(G/H)$, satisfying 
\begin{equation}\label{conv.m}
\lambda\ast_{G/H}\lambda'(\psi)=\int_{G/H}\int_{G/H}\left(\int_H\psi(xhyH)dh\right)d\lambda(xH)d\lambda'(yH),
\end{equation}
for all $\psi\in\mathcal{C}_0(G/H)$, where 
\[
\lambda\ast_{G/H}\lambda'(\psi)=\int_{G/H}\psi(gH)d\lambda\ast_{G/H}\lambda'(gH).
\]
Then, the mapping 
$$\ast_{G/H}:M(G/H)\times M(G/H)\to M(G/H),$$ 
defined by 
$$(\lambda,\lambda')\mapsto\lambda\ast_{G/H}\lambda',$$ 
is a bilinear product. 

The Banach measure space $M(G/H)$ is a Banach algebra with respect to the convolution given by (\ref{conv.m}). Also, $\varphi\mapsto\mu_\varphi$ defines an isometric homomorphism embedding of the Banach function algebra  $L^1(G/H,\mu)$ into the Banach measure algebra $M(G/H)$.

\begin{remark}
Let $G$ be a locally compact group and $H$ be a compact normal subgroup of $G$.

(i) For $\psi\in\mathcal{C}_0(G/H)$ and $x,y\in G$ we can write 
\begin{align*}
\int_H\psi(xhyH)dh
&=\int_H\psi(xyy^{-1}hyH)dh
\\&=\int_H\psi(xyH)dh=\psi(xyH).
\end{align*}
(ii) For $\lambda,\lambda'\in M(G/H)$ and $\psi\in\mathcal{C}_0(G/H)$ we get 
\begin{align*}
\lambda\ast_{G/H}\lambda'(\psi)
&=\int_{G/H}\int_{G/H}\left(\int_H\psi(xhyH)dh\right)d\lambda(xH)d\lambda'(yH)
\\&=\int_{G/H}\int_{G/H}\psi(xyH)d\lambda(xH)d\lambda'(yH)
\\&=\int_{G/H}\int_{G/H}\psi(xHyH)d\lambda(xH)d\lambda'(yH).
\end{align*}
Thus, invoking (\ref{m.conv.G}), we deduce that the convolution defined by (\ref{conv.m}) 
coincides with the canonical convolution of complex measures 
over the locally compact quotient group $G/H$, if $H$ is normal in $G$.
\end{remark}

\begin{remark}
The convolution given by (\ref{conv.m}) introduced in Equation (4.2) of \cite{Jav.Tav} and  studied with different notions in some directions.  
\end{remark}
\begin{remark}
The convolution given by (\ref{conv.m}) introduced in 
Theorem 2.9 of \cite{Derk} and studied with different notions in some directions.
\end{remark}

The following observations are straightforward. 

\begin{proposition}\label{conv.m.q}
{\it Let $G$ be a locally compact group and $H$ be a compact subgroup of $G$.
Let $\lambda,\lambda'\in M(G/H)$. Then we have 
\begin{enumerate}
\item $(\lambda\ast_{G/H}\lambda')_q=\lambda_q\ast_G\lambda'_q$.
\item $(\lambda\ast_{G/H}\lambda')_x=\lambda_x\ast_{G/H}\lambda'$.
\end{enumerate}
}\end{proposition}

\begin{corollary}\label{nn1}
{\it Let $G$ be a locally compact group and $H$ be a compact subgroup of $G$. 
Let $\nu\in M(G)$ and $\lambda,\lambda'\in M(G/H)$. Then
\begin{enumerate}
\item $T_H(\lambda_q\ast_G\nu)=\lambda\ast_{G/H}T_H(\nu)$.
\item $T_H(\lambda_q\ast_G\lambda_q')=\lambda\ast_{G/H}\lambda'$.
\end{enumerate}
}\end{corollary}

\begin{corollary}\label{nn'2}
{\it Let $G$ be a locally compact group and $H$ be a compact subgroup of $G$. 
Let $\nu,\nu'\in M(G)$ with $\nu'\in\mathcal{M}(G:H)$ and $\psi\in\mathcal{C}_0(G/H)$. 
Then
\begin{equation*}
\int_{G/H}\psi(xH)dT_H(\nu\ast_G\nu')(xH)
=\int_G\int_G\left(\int_H\psi(xhyH)dh\right)d\nu(x)d\nu'(y),
\end{equation*}
and
\begin{equation*}
\int_{G/H}\psi(xH)dT_H(\nu\ast_G\nu')(xH)
=\int_G\int_{G/H}\left(\int_H\psi(xhyH)dh\right)dT_H(\nu')(yH)d\nu(x).
\end{equation*}
}\end{corollary}
 
\begin{proposition}
{\it Let $G$ be a locally compact group and $H$ be a compact subgroup of $G$. 
Let $\mu$ be the normalized $G$-invariant measure on $G/H$ associated to (\ref{TH.m})
with respect to the left Haar measure $\sigma$ of $G$ and the 
probability measure of $H$. Also, let 
$\varphi,\varphi'\in L^1(G/H,\mu)$. We then have 
\[
\mu_\varphi\ast_{G/H}\mu_{\varphi'}=\mu_{\varphi\ast_{G/H}\varphi'}.
\]
}\end{proposition}

We then continue the paper by extending the notion of involution for complex 
measures defined on the homogeneous spaces of compact subgroups in locally compact groups, using the operator theoretic approach presented in \cite{AGHF.BBMS-NS, AGHF.IntJM, AGHF.IJM}. 

Let $\lambda\in M(G/H)$. Then
\[
\psi\mapsto \int_{G/H}\left(\int_H\psi(h^{-1}x^{-1}H)dh\right)d\overline{\lambda}(xH),
\] 
defines a uniform continuous linear functional on $\mathcal{C}_0(G/H)$. Indeed, using compactness of $H$, we get 
\begin{align*}
\left|\int_{G/H}\left(\int_H\psi(h^{-1}x^{-1}H)dh\right)d\overline{\lambda}(xH)\right|
&\le \int_{G/H}\left|\int_H\psi(h^{-1}x^{-1}H)dh\right|d|\overline{\lambda}|(xH)
\\&\le \int_{G/H}\int_H|\psi(h^{-1}x^{-1}H)|dhd|\overline{\lambda}|(xH)
\\&\le\int_{G/H}\int_H\|\psi\|_{\sup}dhd|\overline{\lambda}|(xH)
\\&=\|\psi\|_{\sup}\cdot\left(\int_{G/H}d|\overline{\lambda}|(xH)\right)
\\&=\|\psi\|_{\sup}\cdot\|\lambda\|_{M(G/H)}.
\end{align*}

Thus, by Riesz-Markov theorem, 
there exists a complex Radon measure, denoted by 
$\lambda^{\ast^{G/H}}\in M(G/H)$, satisfying 
\begin{equation}\label{inv.m}
\int_{G/H}\psi(xH)d\lambda^{\ast_{G/H}}(xH)
=\int_{G/H}\left(\int_H\psi(h^{-1}x^{-1}H)dh\right)d\overline{\lambda}(xH),
\end{equation}
for all $\psi\in\mathcal{C}_0(G/H)$.

Then, it is easy to see that the mapping $$^{\ast^{G/H}}:M(G/H)\to M(G/H),$$ defined by 
$$\lambda\mapsto\lambda^{\ast^{G/H}},$$ 
is conjugate linear.

Invoking (\ref{m.involution.G}), we deduce that the involution defined by (\ref{inv.m}) 
coincides with the canonical involution of complex measures 
over the compact quotient group $G/H$ if $H$ is normal in $G$.

Next we study some properties of the involution (\ref{inv.m}).
The following propositions presents some algebraic observations concerning the operator theoretic tool $J$. 

\begin{proposition}\label{inv.m.p1}
{\it Let $G$ be a locally compact group and $H$ be a compact subgroup of $G$.
Also, let $\lambda\in M(G/H)$. We then have 
\begin{enumerate}
\item $\lambda^{\ast^{G/H}}(\psi)=\overline{\lambda(\psi^{\ast^{G/H}})}$ for all $\psi\in\mathcal{C}_c(G/H)$.
\item $\lambda^{\ast^{G/H}\ast^{G/H}}=\lambda\circ J_0$.
\end{enumerate}
}\end{proposition}
\begin{proof}
Let $\lambda\in M(G/H)$. (1) For $\psi\in\mathcal{C}_c(G/H)$ we have 
\begin{align*}
\overline{\lambda(\psi^{\ast^{G/H}})}
&=\left(\int_{G/H}\psi^{\ast^{G/H}}(xH)d\lambda(xH)\right)^{-}
\\&=\left(\int_{G/H}T_H(\psi_q^{*^G})(xH)d\lambda(xH)\right)^{-}
\\&=\left(\int_{G/H}\left(\int_H\psi_q^{*^G}(xh)dh\right)d\lambda(xH)\right)^{-}
\\&=\left(\int_{G/H}\left(\int_H\overline{\psi_q(h^{-1}x^{-1})}dh\right)d\lambda(xH)\right)^{-}
\\&=\left(\int_{G/H}\left(\int_H\overline{\psi(h^{-1}x^{-1}H)}dh\right)d\lambda(xH)\right)^{-}
\\&=\int_{G/H}\left(\int_H\psi(h^{-1}x^{-1}H)dh\right)d\overline{\lambda}(xH)
=\lambda^{\ast^{G/H}}(\psi).
\end{align*}
(2) Let $\psi\in\mathcal{C}_c(G/H)$. Then, using (1) and also Proposition \ref{inv.basic}(1), 
we have 
\begin{align*}
\lambda^{\ast^{G/H}\ast^{G/H}}(\psi)
&=\overline{\lambda^{\ast^{G/H}}(\psi^{\ast^{G/H}})}
\\&=\lambda(\psi^{\ast^{G/H}\ast^{G/H}})
\\&=\lambda(J\psi)=\lambda(J_0\psi).
\end{align*}
Invoking, density of $\mathcal{C}_c(G/H)$ in $\mathcal{C}_0(G/H)$ with respect to the uniform norm and also
since $\lambda^{\ast^{G/H}\ast^{G/H}},\lambda\circ J_0$ are uniformly continuous (bounded) linear functional 
on $\mathcal{C}_0(G/H)$, we deduce that
\[
 \lambda^{\ast^{G/H}\ast^{G/H}}(\psi)=\lambda\circ J_0(\psi),
\]
for all $\psi\in\mathcal{C}_0(G/H)$, which completes the proof.
\end{proof}

\begin{proposition}\label{inv.m.p2}
{\it Let $G$ be a locally compact group and $H$ be a compact subgroup of $G$.
Also, let $\lambda\in M(G/H)$. Then,
\begin{enumerate}
\item $T_H(\lambda_q^{*^G})=\lambda^{\ast^{G/H}}$.
\item $\|\lambda^{\ast^{G/H}}\|_{M(G/H)}\le \|\lambda\|_{M(G/H)}$.
\item If $\lambda\in\mathcal{M}(G/H)$ we have $\|\lambda^{\ast^{G/H}}\|_{M(G/H)}=\|\lambda\|_{M(G/H)}$.
\end{enumerate}
}\end{proposition}
\begin{proof}
Let $\lambda\in M(G/H)$. (1) Using Proposition \ref{inv.m.p1}, 
for $\psi\in\mathcal{C}_0(G/H)$, we have 
\begin{align*}
T_H(\lambda_q^{*^G})(\psi)
&=\int_G\psi_q(x)d\lambda_q^{*^G}(x)
\\&=\int_G\psi_q(x^{-1})d\overline{\lambda_q}(x)
\\&=\left(\int_G\overline{\psi_q(x^{-1})}d\lambda_q(x)\right)^{-}
\\&=\left(\int_{G/H}T_H(\psi_q^{*^G})(xH)d\lambda(xH)\right)^{-}
=\lambda^{\ast^{G/H}}(\psi).
\end{align*}
(2) Using (1), Equation (\ref{TH.M.norm}), and also Proposition \ref{q.lam} we have 
\begin{align*}
\|\lambda^{\ast^{G/H}}\|_{M(G/H)}
&=\|T_H(\lambda_q^{*^G})\|_{M(G/H)}
\\&\le \|\lambda_q^{\ast^G}\|_{M(G)}
\\&=\|\lambda_q\|_{M(G)}=\|\lambda\|_{M(G/H)}.
\end{align*}
(3) Let $\lambda\in\mathcal{M}(G/H)$. Then, using (2) and also Proposition \ref{inv.m.p1}, we have 
\[
\|\lambda\|_{M(G/H)}=\|\lambda^{\ast^{G/H}\ast^{G/H}}\|_{M(G/H)}
\le\|\lambda^{\ast^{G/H}}\|_{M(G/H)},
\]
which implies that $\|\lambda^{\ast^{G/H}}\|_{M(G/H)}=\|\lambda\|_{M(G/H)}$.
\end{proof}

Next we shall show that involution of complex measures is anti-homomorphism in some sense.

\begin{proposition}\label{inv.p3}
{\it Let $G$ be a locally compact group and $H$ be a compact subgroup of $G$.
Also, let $\lambda,\lambda'\in M(G/H)$. Then
\[
(\lambda\ast_{G/H}\lambda')^{\ast^{G/H}}=\lambda'^{\ast^{G/H}}\ast_{G/H}\lambda^{\ast^{G/H}}
\]
}\end{proposition}
\begin{proof}
Let $\lambda,\lambda'\in M(G/H)$. Then, using Proposition \ref{conv.m.q}, we have 
\begin{align*}
(\lambda\ast_{G/H}\lambda')^{\ast^{G/H}}
&=T_H\left((\lambda\ast_{G/H}\lambda')_q^{\ast^{G}}\right)
\\&=T_H\left((\lambda_q\ast_{G}\lambda'_q)^{\ast^{G}}\right)
=T_H\left({\lambda'}_q^{\ast^{G}}\ast_{G}\lambda_q^{\ast^{G}}\right).
\end{align*}
Now let $\psi\in\mathcal{C}_0(G/H)$. 
Since $\nu':={\lambda}_q^{\ast^G}\in\mathcal{M}(G:H)$, using Corollary \ref{nn'2}, 
we deduce that 
\begin{align*}
&T_H\left({\lambda'}_q^{\ast^{G}}\ast_{G}\lambda_q^{\ast^{G}}\right)(\psi)
\\&=\int_{G/H}\psi(gH)dT_H({\lambda'}_q^{\ast^{G}}\ast_{G}\lambda_q^{\ast^{G}})(gH)
\\&=\int_{G}\left(\int_{G/H}\psi(xyH)dT_H(\lambda_q^{\ast^G})(yH)\right)d{\lambda'}_q^{\ast^G}(x)
\\&=\int_{G}\left(\int_{G/H}\left(\int_H\psi(xhyH)dh\right)d\lambda^{\ast^{G/H}}(yH)\right)d{\lambda'}_q^{\ast^G}(x)
\\&=\int_{G}\left(\int_{G/H}\left(\int_H\psi(x^{-1}hyH)dh\right)
d\lambda^{\ast^{G/H}}(yH)\right)d\overline{\lambda'_q}(x)
\\&=\int_{G/H}\left(\int_H\left(\int_{G/H}\left(\int_H\psi((xt)^{-1}hyH)dh\right)
d\lambda^{\ast^{G/H}}(yH)\right)dt\right)d\overline{\lambda'}(xH)
\\&=\int_{G/H}\left(\int_H\left(\int_{G/H}\left(\int_H\psi(t^{-1}x^{-1}hyH)dh\right)
d\lambda^{\ast^{G/H}}(yH)\right)dt\right)d\overline{\lambda'}(xH)
\\&=\int_{G/H}\int_{G/H}\left(\int_{H}\psi(xhyH)dh\right)
d\lambda'^{\ast^{G/H}}(xH)d\lambda^{\ast^{G/H}}(yH)
=\lambda'^{\ast^{G/H}}\ast_{G/H}\lambda^{\ast^{G/H}}(\psi).
\end{align*}
Thus, we achieve 
\[
(\lambda\ast_{G/H}\lambda')^{\ast^{G/H}}=T_H\left({\lambda'}_q^{\ast^{G}}\ast_{G}\lambda_q^{\ast^{G}}\right)
=\lambda'^{\ast^{G/H}}\ast_{G/H}\lambda^{\ast^{G/H}}.
\]
\end{proof}

We then conclude the following result. 

\begin{theorem}
{\it Let $G$ be a locally compact group and $H$ be a compact subgroup of $G$.
Then, the Banach measure space 
$\mathcal{M}(G/H)$ is a Banach $*$-algebra with respect to the convolution given by (\ref{conv.m}) and 
also the involution given by (\ref{inv.m}).
}\end{theorem}
\begin{proof}
It is easy to check that $\mathcal{M}(G/H)$ is a closed sub-algebra of $M(G/H)$.
Also, by Propositions \ref{main.m.G1} and \ref{inv.m.p2}, we have 
\[
\lambda^{\ast^{G/H}\ast^{G/H}}=\lambda,
\]
for all $\lambda\in \mathcal{M}(G/H)$. Then, Proposition \ref{inv.p3}, implies that the Banach measure space 
$\mathcal{M}(G/H)$ is a Banach $*$-algebra with respect to the convolution given by (\ref{conv.m}) and 
also the involution given by (\ref{inv.m}). 
\end{proof}

Next we shall show that involution of functions is compatible 
with involution of measures.
 
\begin{proposition}
{\it Let $G$ be a locally compact group and $H$ be a compact subgroup of $G$. 
Let $\mu$ be the normalized $G$-invariant measure on $G/H$ associated to (\ref{TH.m})
with respect to the fixed left Haar measure $\sigma$ of $G$ and the 
probability measure of $H$, and also $\varphi\in L^1(G/H,\mu)$. Then 
\[
(\mu_\varphi)^{\ast^{G/H}}=\mu_{\varphi^{\ast^{G/H}}}.
\]
}\end{proposition}
\begin{proof}
Let $\psi\in\mathcal{C}_0(G/H)$. Then, using Weil's formula, we have 
\begin{align*}
\overline{(\mu_\varphi)^{\ast^{G/H}}(\psi)}&=\mu_{\varphi}(\psi^{\ast^{G/H}})
\\&=\int_{G/H}T_H(\psi_q^{*^G})(xH)\varphi(xH)d\mu(xH)
\\&=\int_{G/H}T_H(\psi_q^{*^G}\cdot\varphi_q)(xH)d\mu(xH)
\\&=\int_G\psi_q^{\ast^G}(x)\varphi_q(x)dx
\\&=\int_G\overline{\psi_q(x)}\varphi_q(x^{-1})dx
\\&=\int_{G/H}\overline{\psi(xH)}\left(\int_H\varphi(h^{-1}x^{-1}H)dh\right)d\mu(xH)
=\overline{\mu_{\varphi^{\ast^{G/H}}}(\psi)},
\end{align*}
which completes the proof.
\end{proof}

\begin{corollary}
{\it Let $G$ be a locally compact group and $H$ be a compact subgroup of $G$.
Let $\mu$ be the normalized $G$-invariant measure over $G/H$ associated to (\ref{TH.m}) with respect to 
left Haar measure $\sigma$ of $G$ and the probability measure of $H$. Then, 
 $\varphi\mapsto\mu_\varphi$ defines an isometric $*$-homomorphism embedding of the Banach function $*$-algebra  
$A^1(G/H,\mu)$ into the Banach measure $*$-algebra $\mathcal{M}(G/H)$.
}\end{corollary}

\bibliographystyle{amsplain}

\end{document}